\documentclass[10pt,leqno]{article}%
\usepackage{amsmath}
\usepackage{amsfonts}
\usepackage{amssymb}
\usepackage{graphicx}%
\newtheorem{theorem}{Theorem}[section]

\newtheorem{remark}[theorem]{Remark}
\newtheorem{example}[theorem]{Example}

\newtheorem{corollary}[theorem]{Corollary}

\newtheorem{lemma}[theorem]{Lemma}

\newenvironment{proof}[1][Proof]{\textbf{#1.} }{\ \rule{0.5em}{0.5em}}

\begin{document}

\title{Application of Mountain Pass Theorem to superlinear equations with fractional
Laplacian controlled by distributed parameters and boundary data}
\author{Dorota Bors\\{\small Faculty of Mathematics and Computer Science}\\{\small University of Lodz}\\{\small ul. S. Banacha 22, 90-238 \L \'{o}d\'{z}, Poland}\\{\small email: bors@math.uni.lodz.pl}}
\date{}
\maketitle

\begin{abstract}
In the paper we consider a boundary value problem involving a differential
equation with\ the fractional Laplacian $(-\Delta)^{\alpha/2}$ for $\alpha
\in\left(  1,2\right)  $ and some superlinear and subcritical nonlinearity
$G_{z}$ provided with a nonhomogeneous Dirichlet exterior boundary condition.
Some sufficient conditions under which the set of weak solutions to the
boundary value problem is nonempty and depends continuously in the
Painlev\'{e}-Kuratowski sense on distributed parameters and exterior boundary
data are stated. The proofs of the existence results rely on the Mountain Pass
Theorem. The application of the continuity results to some optimal control
problem is also provided. \vspace{0.1cm} \newline\textit{Key words and
phrases.} boundary value problems, fractional Laplacian, stability, mountain
pass theorem. \vspace{0.2cm} \newline2010 \textit{Mathematics Subject
Classification. }34A08, 35A15, 35B30, 93D05.

\end{abstract}

\section{Introduction}

The problems with the fractional Laplacian attracted in recent years a lot of
attention as they naturally arise in various areas of applications. The
fractional Laplacian naturally appears in probabilistic framework as well as
in mathematical finance as infinitesimal generators of stable L\'{e}vy
processes \cite{App, BogByc1, BogByc2, BogBycKul, Val}. One can find the
problems involving the fractional Laplacian in mechanics and in elastostatics,
to mention only, a Signorini obstacle problem originating from linear
elasticity \cite{BerSag, Caf, CafSalSil}. Then concerning fluid mechanics and
hydrodynamics the nonlocal fractional Laplacian appears, for instance, in the
quasi-geostrophic fractional Navier-Stokes equation \cite{CafVas} and in the
hydrodynamic model of the flow in some porous media \cite{BonVaz, NocOtaSal,
Vaz, Vaz2}.

In the paper we consider problems modelled by the differential equation
with\ the fractional Laplacian $(-\Delta)^{\alpha/2}$ and some nonlinearity
$G_{z}$ of the form%
\begin{equation}
(-\Delta)^{\alpha/2}z\left(  x\right)  =G_{z}\left(  x,z\left(  x\right)
,u\left(  x\right)  \right)  \text{ in }\Omega\label{1.1}%
\end{equation}
with the nonhomogeneous Dirichlet exterior condition%
\begin{equation}
z\left(  x\right)  =v\left(  x\right)  \text{ in }\mathbb{R}^{n}%
\backslash\Omega\label{1.2}%
\end{equation}
where $\alpha\in\left(  1,2\right)  $ is fixed, $\Omega\subset\mathbb{R}^{n}$
for $n\geq2$ is a bounded domain with a Lipschitz boundary, $G_{z}$ is the
partial derivative of the function $G$ with respect to $z$ variable which is a
suitable Carath\'{e}odory function, $u:\Omega\rightarrow\mathbb{R}^{m},$
$m\geq1,$ is a distributed parameter, $v:\mathbb{R}^{n}\rightarrow\mathbb{R}$
is boundary data and the fractional Laplace operator $(-\Delta)^{\alpha/2}$ is
defined like, for example, in \cite{SerVal4} as%
\begin{equation}
(-\Delta)^{\alpha/2}z\left(  x\right)  =c\left(  n,\alpha\right)
\int\limits_{\mathbb{R}^{n}}\frac{2z\left(  x\right)  -z\left(  x+y\right)
-z\left(  x-y\right)  }{\left\vert y\right\vert ^{n+\alpha}}%
dy\label{Laplacian_i}%
\end{equation}
where $c\left(  n,\alpha\right)  $ is a positive normalizing constant
depending only on $\alpha$ and $n$ like in \cite{BogBycKul,CabSir1,Vaz2}
defined as $c\left(  n,\alpha\right)  =\Gamma\left(  (n+\alpha)/2\right)
/\left(  \left\vert \Gamma\left(  -\alpha/2\right)  \right\vert \pi
^{n/2}2^{1-\alpha}\right)  =\alpha\Gamma\left(  (n+\alpha)/2\right)  /\left(
\Gamma\left(  1-\alpha/2\right)  \pi^{n/2}2^{2-\alpha}\right)  $, cf. also
other normalizations as in \cite{NezPalVal}. By using variational methods in
an appropriate abstract framework developed by Servadei and Valdinoci
\cite{SerVal2}, first of all, we prove the existence results to $\left(
\ref{1.1}\right)  -\left(  \ref{1.2}\right)  $ for a certain class of boundary
data and distributed parameters. Without going into details we examine the
existence of the weak solution $z$ of $\left(  \ref{1.1}\right)  -\left(
\ref{1.2}\right)  $ such that $z-v\in X_{0}$ where $v\in\mathcal{V\subset
}X\cap L^{2}\left(  \mathbb{R}^{n}\right)  $ and $u\in\mathcal{U}\subset
L^{\infty}.$ Next, we address the stability issue for problem $\left(
\ref{1.1}\right)  -\left(  \ref{1.2}\right)  .$ By stability here we mean the
continuous dependence of solutions $z$ on distributed parameters $u$ and
boundary data $v$. It is possible to prove that under some suitable
assumptions, for an arbitrary pair $\left(  u,v\right)  $ there exists a weak
solution $z_{u,v}$ to problem $\left(  \ref{1.1}\right)  -\left(
\ref{1.2}\right)  $ which is stable with respect to the distributed parameters
$u$ and the boundary data $v.$ In general, a weak solution is not unique and
therefore by stability here we understand upper semicontinous dependence of
sets of weak solutions $\mathcal{S}_{u,v}^{z}$ to problem $\left(
\ref{1.1}\right)  -\left(  \ref{1.2}\right)  $ on distributed parameters $u$
and boundary data $v.$ In other words, we prove that $z_{u,v}\rightarrow
z_{u_{0},v_{0}}$ in $X\cap L^{2}\left(  \mathbb{R}^{n}\right)  ,$ if solutions
are unique, which means that $\emptyset\neq\mathrm{Lim}\sup\mathcal{S}%
_{u,v}^{z}\subset\mathcal{S}_{u_{0},v_{0}}^{z}$ in $X\cap L^{2}\left(
\mathbb{R}^{n}\right)  ,$ in general case, provided that $u$ tends to $u_{0}$
in $L^{\infty}$ and $v$ tends to $v_{0}$ in $X\cap L^{2}\left(  \mathbb{R}%
^{n}\right)  .$ The main stability result for problem $\left(  \ref{1.1}%
\right)  -\left(  \ref{1.2}\right)  $ is a direct consequence of Theorem
\ref{twierdzenie11} presented in Section \ref{mainresult}.

It should be noted that the weak formulation of system $\left(  \ref{1.1}%
\right)  $ with homogeneous exterior boundary condition\ corresponds to the
Euler-Lagrange equation for the following integral functional%
\begin{equation}
F\left(  z\right)  =\dfrac{c\left(  n,a\right)  }{2}\int\limits_{\mathbb{R}%
^{2n}}\frac{\left\vert z\left(  x\right)  -z\left(  y\right)  \right\vert
}{\left\vert x-y\right\vert ^{n+\alpha}}^{2}dxdy-\int\limits_{\Omega}G\left(
x,z\left(  x\right)  ,u\left(  x\right)  \right)  dx \label{1.3}%
\end{equation}
where $z\in X_{0},$ cf. \cite{SerVal}$.$ The above functional is referred to
as the functional of action or the functional of energy. On the function $G$
we impose, besides some technical, growth and regularity assumptions, the
following superlinearity assumption%
\begin{equation}
a<pG\left(  x,z,u\right)  \leq zG_{z}\left(  x,z,u\right)  \label{1.4}%
\end{equation}
which is satisfied for some $a>0,$ $p>2$ and $\left\vert z\right\vert $
sufficiently large. This condition guarantees that problem $\left(
\ref{1.1}\right)  $ $-\left(  \ref{1.2}\right)  $ can be referred to as a
superlinear exterior boundary value problem and as illustrated in Remark
\ref{uwaga} the nonlinear functional $\left(  \ref{1.3}\right)  ,$ in general,
can\ be unbounded from above and below. For that reason we cannot adopt the
approach to the existence and stability issue of Dirichlet problem involving
the fractional Laplacian presented for example in \cite{Bor2} where the
coercive functional bounded from below was studied, while in \cite{FelKasVoi}
only the linear case was treated.

In general, in the theory of boundary value problems and its applications we
consider, first of all, the problem of the existence of a solution and next
questions of stability, uniqueness, smoothness, asymptotics etc.. The problem
of existence of solutions to equation $\left(  \ref{1.1}\right)  $ with the
homogenous Dirichlet boundary condition corresponding to critical point of
mountain pass type was considered for example in the recent papers
\cite{FerGueZha,SerVal,SerVal2}. For more references on the existence results
for problems involving nonlocal fractional Laplacian equation with subcritical
nonlinearities, see, for example \cite{Ser} as well as \cite{BarColPabSan,
Cap} for problems with critical nonlinearities. Moreover, the asymptotically
linear case was investigated in \cite{FisSerVal} whereas in \cite{MolSer} one
can find a bifurcation result in the fractional setting. We also refer the
interested reader to papers \cite{AutPuc,BonVaz, CabSir1, CabSir2, CabTan,
CafSil, FelKasVoi, KulSta, SerVal3} for other results related to the
fractional Laplacian. In the present paper we apply to the functional defined
in $\left(  \ref{1.3}\right)  $ the renowned Mountain Pass Theorem presented,
for example, in \cite{MawWil,Rab,Str} which enables us to obtain the existence
result for problem related to $\left(  \ref{1.1}\right)  -\left(
\ref{1.2}\right)  $ similarly as in \cite{SerVal,SerVal2}.

As far as the continuous dependence results of solutions on parameters and
boundary data for equation $\left(  \ref{1.1}\right)  $ are concerned, up to
our best knowledge, the subject in fractional setting seems to have received
almost no attention in the literature. Some continuous dependence results for
homogenous Dirichlet boundary problem involving the fractional Laplacian one
can find in \cite{Bor2} where coercive case is examined by the direct method
of calculus of variations. Differentiable continuous dependence on parameters,
or in other words robustness result are presented in \cite{Bor3} where the
application of theorem on diffeomorphism leads to the stability result for the
problem involving one-dimensional fractional Laplacian with zero boundary
condition. In the present paper we obtain the existence and the continuous
dependence results for the exterior boundary value problem involving the
equation with the fractional Laplacian by adopting the approach presented in
\cite{Bor1} were superlinear elliptic boundary value with the nonhomogeneous
Dirichlet boundary condition was examined.

The structure of the paper reads as follows. Section 2 contains some useful
information on functional spaces introduced in \cite{SerVal2} by Servadei and
Valdinoci with an appropriate extension. The variational formulation of the
problem and some standing assumptions are presented in Section 3, whereas in
Section 4, our attention is focused on proving some auxiliary lemmas which are
of a paramount importance to the rest of the paper. Some sufficient condition
for the existence and continuous dependence of solutions to the exterior
boundary value problem involving the equation with the fractional Laplacian on
distributed parameters and boundary data can be found in Section 5. Finally,
the application of the stability result leads to the existence of the optimal
solution to the control problem described by $\left(  \ref{1.1}\right)
-\left(  \ref{1.2}\right)  $ with a integral performance index expressed by
some cost functional as asserted in Theorem \ref{twierdzenie5.1} and Theorem
\ref{twierdzenie5.2} in Section 6. The proof of these theorems relies in an
essential way on the continuous dependence results from Section 5.

\section{Functional setup\label{fra_set}}

In this section we introduce the notation and give some preliminary results
which will be useful in the sequel. We now recall, following
\cite{BarColPabSan, BraColPabSan, NezPalVal, RosSer1, RosSer2}, the definition
of the classical fractional Sobolev space. Let $D$ be an open set in
$\mathbb{R}^{n}.$ For $\alpha\in\left(  1,2\right)  ,$ by $H^{\alpha/2}\left(
D\right)  $, we denote the following space%
\begin{equation}
H^{\alpha/2}\left(  D\right)  =\left\{  z\in L^{2}\left(  D\right)
:\frac{z\left(  x\right)  -z\left(  y\right)  }{\left\vert x-y\right\vert
^{\left(  n+\alpha\right)  /2}}\in L^{2}\left(  D\times D\right)  \right\}
\label{definicja1}%
\end{equation}
where $D\subset\mathbb{R}^{n}$ for $n>\alpha$ is a general, possibly
unbounded, open domain in $\mathbb{R}^{n}$ with suitably smooth boundary, for
example Lipschitz (in our case $D=\Omega$ or $D=\mathbb{R}^{n})$. In the
literature, fractional Sobolev spaces are also referred to as Aronszajn,
Gagliardo or Slobodeckij\textbf{ }spaces, associated with the names of the
ones who introduced them almost simultaneously, see \cite{Aro,Gag,Slo}.

\noindent The space $H^{\alpha/2}\left(  D\right)  $ is a Hilbert space placed
between $L^{2}\left(  D\right)  $ and $H^{1}\left(  D\right)  $ endowed with
the norm%
\begin{equation}
\left\Vert z\right\Vert _{H^{\alpha/2}\left(  D\right)  }=\left\Vert
z\right\Vert _{L^{2}\left(  D\right)  }+\left(  \int\limits_{D\times D}%
\frac{\left\vert z\left(  x\right)  -z\left(  y\right)  \right\vert ^{2}%
}{\left\vert x-y\right\vert ^{n+\alpha}}dxdy\right)  ^{1/2}. \label{norma1}%
\end{equation}
$H_{0}^{\alpha/2}\left(  D\right)  $ can be defined as completion of
$C_{0}^{\infty}\left(  D\right)  $ with respect to the norm in $H^{\alpha
/2}\left(  D\right)  $ or $H^{\alpha/2}\left(  \mathbb{R}^{n}\right)  $ and
one can extend the functions from $H_{0}^{\alpha/2}\left(  D\right)  $ with
$0$ to $\mathbb{R}^{n}$ as presented in \cite{FelKasVoi}. It should be
emphasized that for domains with non-Lipschitz boundary or for $\alpha
\in(0,1]$ various definitions of the space of the fractional order might lead
to non-equivalent formulations, for more details see, for example
\cite{BonVaz, FelKasVoi, NocOtaSal}$.$

Due to the nonlocal character of the fractional Laplacian, we will consider
spaces $X^{\alpha/2},$ $X_{0}^{\alpha/2}$ introduced in \cite{SerVal2} and
denoted therein by $X,$ $X_{0},$ respectively. However, in the paper we shall
work with the specific kernel of the form $K\left(  x\right)  =\left\vert
x\right\vert ^{-\left(  n+\alpha\right)  }$. Let $\Omega$ be bounded domain
with a Lipschitz boundary and denote by $Q$ the following set
\[
Q=\mathbb{R}^{2n}\backslash\left(  \left(  \mathbb{R}^{n}\backslash
\Omega\right)  \times\left(  \mathbb{R}^{n}\backslash\Omega\right)  \right)
.
\]
We define%
\[
X^{\alpha/2}=\left\{  z:\mathbb{R}^{n}\rightarrow\mathbb{R}:z|_{\Omega}\in
L^{2}\left(  \Omega\right)  \text{ and }\frac{z\left(  x\right)  -z\left(
y\right)  }{\left\vert x-y\right\vert ^{\left(  n+\alpha\right)  /2}}\in
L^{2}\left(  Q\right)  \right\}
\]
with the norm
\begin{equation}
\left\Vert z\right\Vert _{X^{\alpha/2}}=\left\Vert z\right\Vert _{L^{2}\left(
\Omega\right)  }+[z]_{n,\alpha}=\left\Vert z\right\Vert _{L^{2}\left(
\Omega\right)  }+\left(  \int\limits_{Q}\frac{\left\vert z\left(  x\right)
-z\left(  y\right)  \right\vert ^{2}}{\left\vert x-y\right\vert ^{n+\alpha}%
}dxdy\right)  ^{1/2}. \label{normaX}%
\end{equation}
For the proof that $\left\Vert \cdot\right\Vert _{X^{\alpha/2}}$ is a norm on
$X^{\alpha/2},$ see, for instance, \cite{SerVal2}. Obviously, $Q\varsupsetneq
\Omega\times\Omega$ and it implies that $X^{\alpha/2}$ and $H^{\alpha
/2}\left(  \Omega\right)  $ are not equivalent as the norms $\left(
\ref{norma1}\right)  $ and $\left(  \ref{normaX}\right)  $ are not the same.
We also consider the linear subspace of $X^{\alpha/2}$
\[
X_{0}^{\alpha/2}=\left\{  z\in X^{\alpha/2}:z=0\text{ a.e. in }\mathbb{R}%
^{n}\backslash\Omega\right\}
\]
with the norm
\begin{equation}
\left\Vert z\right\Vert _{X_{0}^{\alpha/2}}=\left(  \int\limits_{Q}%
\frac{\left\vert z\left(  x\right)  -z\left(  y\right)  \right\vert ^{2}%
}{\left\vert x-y\right\vert ^{n+\alpha}}dxdy\right)  ^{1/2}. \label{normaX0}%
\end{equation}
We remark that $X^{\alpha/2},$ $X_{0}^{\alpha/2}$ are nonempty, since, by
\cite[Lemma 11]{SerVal3}, $C_{0}^{2}\left(  \Omega\right)  \subseteq
X_{0}^{\alpha/2}$. Moreover, the space $X_{0}^{\alpha/2}$ is a Hilbert space,
for the proof of this, see,\cite[Lemma 2.3]{FelKasVoi} or \cite[Lemma
7]{SerVal2} and the inner product has the form
\[
\left(  z_{1},z_{2}\right)  _{X_{0}^{\alpha/2}}=\left(  \int\limits_{Q}%
\frac{\left(  z_{1}\left(  x\right)  -z_{1}\left(  y\right)  \right)  \left(
z_{2}\left(  x\right)  -z_{2}\left(  y\right)  \right)  }{\left\vert
x-y\right\vert ^{n+\alpha}}dxdy\right)  ^{1/2}.
\]
Furthermore, we have:

\begin{itemize}
\item $X^{\alpha/2}\subset H^{\alpha/2}\left(  \Omega\right)  ,$ cf.
\cite[Example 2]{FelKasVoi}, \cite[Lemma 5]{SerVal2}

\item $H^{\alpha/2}\left(  \mathbb{R}^{n}\right)  \subset X^{\alpha/2},$ cf.
\cite[Remark 2.2]{FelKasVoi},

\item $X_{0}^{\alpha/2}\subset H^{\alpha/2}\left(  \mathbb{R}^{n}\right)  \cap
H_{0}^{\alpha/2}\left(  \Omega\right)  ,$ cf. \cite[Lemma 5]{SerVal2}%
,\cite[Remark 2.2]{FelKasVoi}, \cite{SerVal5}.
\end{itemize}

In order to consider nonhomogenous exterior boundary data we assume that these
values are prescribed by a function $v:\mathbb{R}^{n}\rightarrow\mathbb{R}.$
For the functional treatment of this problem we need a modification of the
space $X^{\alpha/2}$ that turns this normed space into a Hilbert space with
the appropriate inner product. For that reason we define the following space
\[
Y^{\alpha/2}=X^{\alpha/2}\cap L^{2}\left(  \mathbb{R}^{n}\right)
\]
with the norm%
\begin{equation}
\left\Vert z\right\Vert _{Y^{\alpha/2}}=\left\Vert z\right\Vert _{L^{2}\left(
\mathbb{R}^{n}\right)  }+\left(  \int\limits_{Q}\frac{\left\vert z\left(
x\right)  -z\left(  y\right)  \right\vert ^{2}}{\left\vert x-y\right\vert
^{n+\alpha}}dxdy\right)  ^{1/2}. \label{normaX1}%
\end{equation}
By analogy with the proof of Lemma 2.3 in \cite{FelKasVoi} or the proof of
Lemma 7 in \cite{SerVal2} it can be seen that this space is a separable
Hilbert space with the inner product defined by%
\begin{equation}
\left(  z_{1},z_{2}\right)  _{Y^{\alpha/2}}=\left(  z_{1},z_{2}\right)
_{L^{2}\left(  \mathbb{R}^{n}\right)  }+\left(  z_{1},z_{2}\right)
_{X_{0}^{\alpha/2}}. \label{scalarX1}%
\end{equation}
Immediately, from the definition we have the following inclusions%
\[
H^{\alpha/2}\left(  \mathbb{R}^{n}\right)  \subset Y^{\alpha/2}\subset
X^{\alpha/2}.
\]

\noindent It is worth reminding the reader that for a bounded domain
$\Omega\subset\mathbb{R}^{n}$ with a Lipschitz boundary, the space
$X_{0}^{\alpha/2}$ is compactly embedded into $L^{s}\left(  \Omega\right)  $
for $s\in\left[  1,2_{\alpha}^{\ast}\right)  $ where $2_{\alpha}^{\ast
}=2n/\left(  n-\alpha\right)  $ and the inequality holds%
\begin{equation}
\left\Vert z\right\Vert _{L^{s}\left(  \Omega\right)  }\leq C\left\Vert
z\right\Vert _{X_{0}^{\alpha/2}} \label{embeding}%
\end{equation}
for $n>\alpha$ and any $z\in X_{0}^{\alpha/2}$, cf. Lemma 8 in \cite{SerVal2}
or Corollary 7.2 in \cite{NezPalVal}$.$

\noindent For further details on the fractional Sobolev spaces we refer the
reader to \cite{NezPalVal} and the references therein, while for other details
on $X^{\alpha/2}$ and $X_{0}^{\alpha/2}$ we refer to \cite{SerVal3}, where
these functional spaces were introduced and various properties of these spaces
were proved.

\section{Variational formulation of the problem and standing assumptions}

In the paper we shall consider a problem involving a weak formulation of the
following equation with the fractional Laplacian of the form
\begin{equation}
\left\{
\begin{array}
[c]{l}%
(-\Delta)^{\alpha/2}z\left(  x\right)  =G_{z}\left(  x,z\left(  x\right)
,u\left(  x\right)  \right)  \text{ in }\Omega\subset\mathbb{R}^{n}\\
z\left(  x\right)  =v\left(  x\right)  \text{ in }\mathbb{R}^{n}%
\backslash\Omega,
\end{array}
\right.  \label{2.1}%
\end{equation}
where the exterior boundary condition will be ascertained by claiming that
$z-v\in X_{0}^{\alpha/2},$ $G_{z}:$ $\Omega\times\mathbb{R}\times
\mathbb{R}^{m}\rightarrow\mathbb{R}$ and $\Omega\subset\mathbb{R}^{n},$
$n>\alpha$ is a bounded domain with a Lipschitz boundary. Let $v_{0}$ be a
fixed element from the space $Y^{\alpha/2}.$ By $\mathcal{V}$ we denote the
set of all boundary data $v$ such that
\[
\mathcal{V}=\left\{  v\in Y^{\alpha/2}:\left\Vert v-v_{0}\right\Vert
_{Y^{\alpha/2}}\leq l_{1}\right\}
\]
for $l_{1}>0$ and $\mathcal{U}$ denotes the set of distributed parameters $u$
of the form%
\[
\mathcal{U}=\left\{  u\in L^{\infty}:u\left(  x\right)  \in U\subset
\mathbb{R}^{m}\text{ for a.e. }x\in\Omega\text{ and }\left\Vert u\right\Vert
_{L^{\infty}}\leq l_{2}\right\}
\]
for $l_{2}>0$ and some subset $U$ of $\mathbb{R}^{m}$ with $m\geq1.$

\noindent Besides, one additionally require that the mapping $v$ is chosen
such that $v\in\left(  X_{0}^{\alpha/2}\right)  ^{\perp},$ while we have the
following orthogonal decomposition
\begin{equation}
Y^{\alpha/2}=X_{0}^{\alpha/2}\oplus\left(  X_{0}^{\alpha/2}\right)  ^{\perp}.
\label{suma}%
\end{equation}

As it was announced we look for a weak solution of $\left(  \ref{2.1}\right)
$ such that $z-v\in X_{0}^{\alpha/2}$. Let $w=z-v,$ then the problem in
$\left(  \ref{2.1}\right)  $ can be rewritten in the following homogenized
form%
\begin{equation}
\left\{
\begin{array}
[c]{l}%
(-\Delta)^{\alpha/2}w\left(  x\right)  +\left(  -\Delta\right)  ^{\alpha
/2}v\left(  x\right)  =G_{w}\left(  x,\left(  w+v\right)  \left(  x\right)
,u\left(  x\right)  \right)  \text{ in }\Omega\subset\mathbb{R}^{n}\\
w\left(  x\right)  =0\text{ in }\mathbb{R}^{n}\backslash\Omega.
\end{array}
\right.  \label{problem2}%
\end{equation}
\newline Next, we say that $w\in X_{0}^{\alpha/2}$ is a weak solution or an
energy solution to $\left(  \ref{problem2}\right)  $ if the identity%
\begin{align}
&  c\left(  n,\alpha\right)  \int\limits_{Q}\frac{\left(  w\left(  x\right)
-w\left(  y\right)  \right)  \left(  \varphi\left(  x\right)  -\varphi\left(
y\right)  \right)  }{\left\vert x-y\right\vert ^{n+\alpha}}dxdy+c\left(
n,\alpha\right)  \int\limits_{Q}\frac{\left(  v\left(  x\right)  -v\left(
y\right)  \right)  \left(  \varphi\left(  x\right)  -\varphi\left(  y\right)
\right)  }{\left\vert x-y\right\vert ^{n+\alpha}}dxdy\label{weak}\\
&  =\int\limits_{\Omega}G_{w}\left(  x,\left(  w+v\right)  \left(  x\right)
,u\left(  x\right)  \right)  \varphi\left(  x\right)  dx\nonumber
\end{align}
holds for every function $\varphi\in X_{0}^{\alpha/2}.$ Then the functional of
action defined on $X_{0}^{\alpha/2}$ reads as%
\begin{align}
F_{u,v}\left(  w\right)   &  =c\left(  n,\alpha\right)  \left(  \int
\limits_{Q}\frac{\left\vert w\left(  x\right)  -w\left(  y\right)  \right\vert
^{2}}{2\left\vert x-y\right\vert ^{n+\alpha}}dxdy+\int\limits_{Q}\frac{\left(
v\left(  x\right)  -v\left(  y\right)  \right)  \left(  w\left(  x\right)
-w\left(  y\right)  \right)  }{\left\vert x-y\right\vert ^{n+\alpha}%
}dxdy\right) \label{akcja1}\\
&  -\int\limits_{\Omega}G\left(  x,\left(  w+v\right)  \left(  x\right)
,u\left(  x\right)  \right)  dx\nonumber
\end{align}
and is related to $F$ defined by $\left(  \ref{1.3}\right)  $ by%
\[
F_{u,v}\left(  w\right)  =F\left(  w+v\right)  -c\left(  n,\alpha\right)
\int\limits_{Q}\frac{\left\vert v\left(  x\right)  -v\left(  y\right)
\right\vert ^{2}}{2\left\vert x-y\right\vert ^{n+\alpha}}dxdy.
\]

On the function $G$ we shall impose the following conditions:

\begin{enumerate}
\item[(C1)] $G,$ $G_{z}$ are Carath\'{e}odory functions, i.e. they are
measurable with respect to $x$ for any $(z,u)\in\mathbb{R}\times\mathbb{R}%
^{m}$ and continuous with respect to $(z,u)\in\mathbb{R}\times\mathbb{R}^{m}$
for a.e. $x\in\Omega;$

\item[(C2)] for any bounded subset $U_{0}\subset U,$ there exists $c>0$ such
that
\[
\left\vert G\left(  x,z,u\right)  \right\vert \leq c\left(  1+\left\vert
z\right\vert ^{s}\right)  ,\,\left\vert G_{z}\left(  x,z,u\right)  \right\vert
\leq c\left(  1+\left\vert z\right\vert ^{s-1}\right)  ,
\]
for $z\in\mathbb{R},$ $u\in U_{0}$ and a.e. $x\in\Omega$ , where $s\in\left(
2,2_{\alpha}^{\ast}\right)  $ with $2_{\alpha}^{\ast}=2n/(n-\alpha)$ for
$n>\alpha~$and $\alpha\in\left(  1,2\right)  ;$

\item[(C3)] there exist $p>2,$ $a>0$ and $R>0$ such that
\[
a<pG\left(  x,z,u\right)  \leq zG_{z}\left(  x,z,u\right)
\]
for a.e. $x\in\Omega$ , any $u\in U$ and $\left\vert z\right\vert \geq R;$

\item[(C4)] there exist $\zeta>0$ and $0<b<\frac{c\left(  n,\alpha\right)
}{2}$ such that
\[
\left\vert G\left(  x,z,u\right)  +\frac{c\left(  n,\alpha\right)  }{2}%
z^{2}\right\vert \leq\frac{b}{2}\left\vert z-v_{0}\left(  x\right)
\right\vert ^{2}%
\]
for $\left\vert z\right\vert \leq\zeta,$ $u\in U$ and a.e. $x\in\Omega$ and
$\mathrm{ess}\sup\left\vert v_{0}\right\vert <\zeta;$

\item[(C5)] for any $u_{0}\in U$ and $\varepsilon>0,$ there exists a constant
$c>0$ such that%
\begin{align*}
\left\vert G\left(  x,z,u_{1}\right)  -G\left(  x,z,u_{2}\right)  \right\vert
&  \leq c\left(  1+\left\vert z\right\vert ^{2}\right)  \left\vert u_{1}%
-u_{2}\right\vert \\
\left\vert G_{z}\left(  x,z,u_{1}\right)  -G_{z}\left(  x,z,u_{2}\right)
\right\vert  &  \leq c\left(  1+\left\vert z\right\vert \right)  \left\vert
u_{1}-u_{2}\right\vert
\end{align*}
for a.e. $x\in\Omega$, any $z\in\mathbb{R}$ and $u_{1},$ $u_{2}\in U$ such
that $\left\vert u_{1}-u_{0}\right\vert <\varepsilon$ and $\left\vert
u_{2}-u_{0}\right\vert <\varepsilon.$
\end{enumerate}

In short, conditions $\left(  C1\right)  -\left(  C4\right)  $, as we shall
demonstrate, guarantee the existence of weak solution to problem $\left(
\ref{2.1}\right)  $ corresponding to the critical points of mountain pass type
of the associated functional of action. If, additionally, condition $(C5)$ is
satisfied, it is feasible to prove that these solutions depend continuously
(or in general case upper semicontinuously) on distributed parameter $u$ and
boundary data $v$ in appropriate topologies.

\section{Verification of Mountain Pass Theorem assumptions}

In this section we focus our attention on proving some auxiliary results which
are of a key importance to the rest of the paper. First of all, we recall some
definitions. Let $\mathcal{I}:E\rightarrow\mathbb{R}$ be a functional of
$C^{1}-$class defined on a real Banach space $E$. A point $w\in E$ is a
critical point of the functional $\mathcal{I}$ if $\mathcal{I}^{\prime}\left(
w\right)  =0$. Moreover, a value $c=\mathcal{I}\left(  w\right)  $ is referred
to as a critical value of the functional $\mathcal{I}$ related to a critical
point $w$.

In what follows we will need some compactness properties of the functional
$\mathcal{I}$ guaranteeing for example by the Palais-Smale condition. Now we
recall what this means. A sequence $\left\{  w_{k}\right\}  \subset E$ is
referred to as a Palais-Smale sequence for a functional $\mathcal{I}$ if for
some $C>0$, any $k\in\mathbb{N},$ $\left\vert \mathcal{I}\left(  w_{k}\right)
\right\vert \leq C$ and $\mathcal{I}^{\prime}\left(  w_{k}\right)
\rightarrow0$ as $k\rightarrow\infty.$ We say that $\mathcal{I}$ satisfies the
Palais-Smale condition if any Palais-Smale sequence possesses a convergent
subsequence in $E$. For more details on the Palais-Smale condition we refer
the reader to Chapter 4.2 in book \cite{MawWil} by Mawhin and Willem.

In this section we shall use, as in \cite{Bor1}, the following version of the
Mountain Pass Theorem, cf. \cite{MawWil,Str}.

\begin{theorem}
If\label{MPT}\newline$(a)$ there exist $\omega_{0},\omega_{1}\in E$ and a
bounded neighborhood $B$ of $\omega_{0},$ such that $\omega_{1}\in
E\setminus\overline{B},$\newline$(b)\inf\nolimits_{\omega\in\partial
B}\mathcal{I}\left(  \omega\right)  >\max\left\{  \mathcal{I}\left(
\omega_{0}\right)  ,\mathcal{I}\left(  \omega_{1}\right)  \right\}  ,$%
\newline$(c)$ $c=\inf\nolimits_{g\in M}\max\nolimits_{t\in\left[  0,1\right]
}\mathcal{I}\left(  g\left(  t\right)  \right)  ,$ where $M=\left\{  g\in
C\left(  \left[  0,1\right]  ,E\right)  :g\left(  0\right)  =\omega_{0},\text{
}g\left(  1\right)  =\omega_{1}\right\}  ,$\newline$(d)$ $\mathcal{I}$
satisfies the Palais-Smale condition,\newline then $c$ is a critical value of
$\mathcal{I}$ and $c>\max\left\{  \mathcal{I}\left(  \omega_{0}\right)
,\mathcal{I}\left(  \omega_{1}\right)  \right\}  .$
\end{theorem}

Throughout this section we shall use the following notation and definitions.
First,%
\[
M_{r}=\left\{  g\in C\left(  \left[  0,1\right]  ,B_{r}\right)  :g\left(
0\right)  =\omega_{0},\text{ }g\left(  1\right)  =\omega_{1}\right\}
\]
where $\omega_{0},\omega_{1}\in B_{r}$ and $B_{r}=\left\{  w\in X_{0}%
^{\alpha/2}:\left\Vert w\right\Vert _{X_{0}^{\alpha/2}}<r\right\}  $ and
$r>0.$ Next, for $k\in\mathbb{N}_{0},$ let
\[
\mathcal{I}_{k}:X_{0}^{\alpha/2}\rightarrow\mathbb{R}%
\]
denote an arbitrary sequence of functionals of $C^{1}-$class$,$ and
$c_{k}\left(  r\right)  $ be the value defined by setting
\begin{equation}
c_{k}\left(  r\right)  =\inf\limits_{g\in M_{r}}\max\limits_{t\in\left[
0,1\right]  }\mathcal{I}_{k}\left(  g\left(  t\right)  \right)  .
\label{(1.12)}%
\end{equation}
Moreover, for $k\in\mathbb{N}_{0}$, let $W_{k}\left(  r\right)  $ denote the
set of all critical points in $B_{r}$ corresponding to the value $c_{k}\left(
r\right)  ,$ i.e.%
\begin{equation}
W_{k}\left(  r\right)  =\left\{  w\in B_{r}:\mathcal{I}_{k}\left(  w\right)
=c_{k}\left(  r\right)  \text{ and }\mathcal{I}_{k}^{\prime}\left(  w\right)
=0\right\}  . \label{(1.13)}%
\end{equation}
In what follows, we shall establish the properties of the upper limit of sets
$W_{k}\left(  r\right)  $ in order to state stability results for the problem
under consideration. Let us recall that by the Painlev\'{e}-Kuratowski upper
limit of sets or, in short, the upper limit of sets $S_{k}$, denoted by
$\mathrm{Lim}\sup S_{k},$ we understand the set of all cluster points with
respect to the strong topology of $E$ of a sequence $\left\{  s_{k}\right\}  $
such that $s_{k}\in S_{k}$ for $k\in\mathbb{N}$. In particular, $\mathrm{Lim}%
\sup W_{k}\left(  r\right)  $ is the upper limit of the sets $W_{k}\left(
r\right)  ,$ $k\in\mathbb{N}$, hence the set of all cluster points with
respect to the strong topology of $X_{0}^{\alpha/2}$ of a sequence $\left\{
w_{k}\right\}  $ such that $w_{k}\in W_{k}\left(  r\right)  $ for
$k\in\mathbb{N}.$ For more details on the Painlev\'{e}-Kuratowski upper limits
of sets we refer the reader to the book \cite{AubFra} by Aubin and Frankowska.

Now we prove, under some assumptions imposed on the sequences $\left\{
\mathcal{I}_{k}\right\}  ,$ $\left\{  \mathcal{I}_{k}^{\prime}\right\}  ,$
that the upper limit in $X_{0}^{\alpha/2}$ of sets $W_{k}\left(  r\right)  $
defined in $\left(  \ref{(1.13)}\right)  $ is nonempty and is a subset of
$W_{0}\left(  r\right)  $.

\begin{lemma}
Assume that\newline$(a)$ for any $k\in\mathbb{N}_{0},$ the functional
$\mathcal{I}_{k},$ is of $C^{1}-$class\newline$(b)$ the
functional\ $\mathcal{I}_{0}$ satisfies the Palais-Smale condition,\newline%
$(c)$ the sequences $\left\{  \mathcal{I}_{k}\right\}  ,$ $\left\{
\mathcal{I}_{k}^{\prime}\right\}  $ tend uniformly on the ball $B_{r}$ to
$\mathcal{I}_{0},$ $\mathcal{I}_{0}^{\prime},$ respectively, \newline$\left(
d\right)  $ for any sufficiently large $k\in\mathbb{N}_{0},$ the sets
$W_{k}\left(  r\right)  $ are nonempty$.$\newline Then any sequence $\left\{
w_{k}\right\}  $ such that $w_{k}\in W_{k}\left(  r\right)  ,$ $k\in
\mathbb{N}$ is relatively compact in $X_{0}^{\alpha/2}$ and $\mathrm{Lim}\sup
W_{k}\left(  r\right)  \subset W_{0}\left(  r\right)  $.$\label{lemat11}$
\end{lemma}

\begin{proof}
In the proof we shall follow the lines of the proof of Lemma 3.1 from
\cite{Bor1}. First of all, one can prove that $\mathrm{Lim}\sup W_{k}\left(
r\right)  $ is not empty$.$ To do this, let $\left\{  w_{k}\right\}  $ be an
arbitrary sequence such that $w_{k}\in W_{k}\left(  r\right)  $ for
$k\in\mathbb{N}_{0}.$ Such a sequence exists by $\left(  d\right)  $.
Moreover, by $\left(  c\right)  ,$ $0=\lim\nolimits_{k\rightarrow\infty
}\mathcal{I}_{0}^{\prime}\left(  w_{k}\right)  $ as $\mathcal{I}_{k}^{\prime
}\left(  w_{k}\right)  =0$ for $k\in\mathbb{N}_{0}.$ Furthermore, $\left\Vert
w_{k}\right\Vert _{X_{0}^{\alpha/2}}<r$ hence the sequence $\mathcal{I}%
_{0}\left(  w_{k}\right)  $ is bounded. Since $\mathcal{I}_{0}$ satisfies the
Palais-Smale condition, as assumed in $(b)$, the sequence $\left\{
w_{k}\right\}  $ is relatively compact in $X_{0}^{\alpha/2}$, that is,
$\mathrm{Lim}\sup W_{k}\left(  r\right)  $ is not empty$.$\newline Next,
again, by $\left(  c\right)  ,$ we get
\begin{equation}
\lim\limits_{k\rightarrow\infty}c_{k}\left(  r\right)  =c_{0}\left(  r\right)
. \label{(1.14)}%
\end{equation}
Moreover for any sequence $\left\{  w_{k}\right\}  $ such that $w_{k}\in
W_{k}\left(  r\right)  $ for $k\in\mathbb{N},$ we have $\mathcal{I}_{0}\left(
w_{k}\right)  -\mathcal{I}_{k}\left(  w_{k}\right)  \rightarrow0$ as
$k\rightarrow\infty.$ From the convergence in $(\ref{(1.14)})$, we conclude
that $\lim\nolimits_{k\rightarrow\infty}\mathcal{I}_{0}\left(  w_{k}\right)
=c_{0}\left(  r\right)  $. Since the set $\mathrm{Lim}\sup W_{k}\left(
r\right)  $ is not empty, choose $\tilde{w}$ from this set, so that $\tilde
{w}$ is a cluster point of some sequence $\left\{  w_{k}\right\}  $ such that
$w_{k}\in W_{k}\left(  r\right)  $ for $k\in\mathbb{N}.$ Therefore, passing to
a subsequence, if necessary, we may assume that $w_{k}\rightarrow\tilde{w}$ as
$k\rightarrow\infty.$ Suppose that $\tilde{w}\notin W_{0}\left(  r\right)  ,$
i.e. $\mathcal{I}_{0}\left(  \tilde{w}\right)  \neq c_{0}\left(  r\right)  $
or $\mathcal{I}_{0}^{\prime}\left(  \tilde{w}\right)  \neq0$. Let us observe
that the second condition is false. Indeed, assumption $\left(  c\right)  $
and the first part of our proof\textbf{\ }allow us to write
\[
\mathcal{I}_{0}^{\prime}\left(  \tilde{w}\right)  =\lim\limits_{k\rightarrow
\infty}\left(  \mathcal{I}_{0}^{\prime}\left(  w_{k}\right)  -\mathcal{I}%
_{k}^{\prime}\left(  w_{k}\right)  \right)  =0.
\]
By putting $\delta=\mathcal{I}_{0}\left(  \tilde{w}\right)  -\mathcal{I}%
_{0}\left(  w_{0}\right)  ,$ where $w_{0}\in W_{0}\left(  r\right)  $ and
$\delta\neq0,$ we arrive at
\[
c_{k}\left(  r\right)  -c_{0}\left(  r\right)  =\left[  \mathcal{I}_{k}\left(
w_{k}\right)  -\mathcal{I}_{0}\left(  w_{k}\right)  \right]  +\left[
\mathcal{I}_{0}\left(  w_{k}\right)  -\mathcal{I}_{0}\left(  \tilde{w}\right)
\right]  +\delta.
\]
From $(\ref{(1.14)})$ and by $\left(  a\right)  $ and $\left(  c\right)  ,$ we
have that $c_{k}\left(  r\right)  -c_{0}\left(  r\right)  \rightarrow0,$
$\mathcal{I}_{k}\left(  w_{k}\right)  -\mathcal{I}_{0}\left(  w_{k}\right)
\rightarrow0$ and $\mathcal{I}_{0}\left(  w_{k}\right)  -\mathcal{I}%
_{0}\left(  \tilde{w}\right)  \rightarrow0$ as $k\rightarrow\infty.$ This
contradicts the fact that $\delta\neq0.$ Thus $\tilde{w}\in W_{0}\left(
r\right)  $ and consequently $\mathrm{Lim}\sup W_{k}\left(  r\right)  \subset
W_{0}\left(  r\right)  ,$ which concludes the proof.
\end{proof}

What we need at this point of our consideration is to examine a specific form
of the functional $\mathcal{I}_{k}$ derived from the functional of action
given by $(\ref{akcja1}).$

Let $\left\{  v_{k}\right\}  $ be a sequence of boundary data and $\left\{
u_{k}\right\}  $ a sequence of parameters such that $\left\{  v_{k}\right\}
\in\mathcal{V}$, $\left\{  u_{k}\right\}  \in\mathcal{U},$ for $k\in
\mathbb{N}_{0}$. Furthermore, let $\left\{  \mathcal{F}_{k}\right\}  $ stand
for the sequence of functionals of the form%
\begin{align}
\mathcal{F}_{k}\left(  w\right)   &  =F_{u_{k},v_{k}}\left(  w\right)
+\int\limits_{\Omega}G\left(  x,v_{k}\left(  x\right)  ,u_{k}\left(  x\right)
\right)  dx\label{akcja2}\\
&  =c\left(  n,\alpha\right)  \left(  \int\limits_{Q}\frac{\left\vert w\left(
x\right)  -w\left(  y\right)  \right\vert ^{2}}{2\left\vert x-y\right\vert
^{n+\alpha}}dxdy+\int\limits_{Q}\frac{\left(  v_{k}\left(  x\right)
-v_{k}\left(  y\right)  \right)  \left(  w\left(  x\right)  -w\left(
y\right)  \right)  }{\left\vert x-y\right\vert ^{n+\alpha}}dxdy\right)
\nonumber\\
&  -\int\limits_{\Omega}G\left(  x,\left(  w+v_{k}\right)  \left(  x\right)
,u_{k}\left(  x\right)  \right)  dx+\int\limits_{\Omega}G\left(
x,v_{k}\left(  x\right)  ,u_{k}\left(  x\right)  \right)  dx\nonumber
\end{align}
for which we define the value
\begin{equation}
c_{k}=\inf\limits_{g\in M}\max\limits_{t\in\left[  0,1\right]  }%
\mathcal{F}_{k}\left(  g\left(  t\right)  \right)  , \label{(1.10)}%
\end{equation}
where
\[
M=\left\{  g\in C\left(  \left[  0,1\right]  ,X_{0}^{\alpha/2}\right)
:g\left(  0\right)  =\omega_{0},\text{ }g\left(  1\right)  =\omega
_{1}\right\}
\]
and $\omega_{0},\omega_{1}\in X_{0}^{\alpha/2}.$

Here and throughout the paper, for $k\in\mathbb{N}_{0},$ let $W_{k}$ denotes
the set of critical points corresponding to the value $c_{k}$, that is, the
set of the form
\begin{equation}
W_{k}=\left\{  w\in X_{0}^{\alpha/2}:\mathcal{F}_{k}\left(  w\right)
=c_{k}\text{ and }\mathcal{F}_{k}^{\prime}\left(  w\right)  =0\right\}  .
\label{(1.11)}%
\end{equation}
In Section \ref{mainresult}, we shall prove that for each $k\in\mathbb{N}$,
the set $W_{k}$ is not empty and the sequence of sets $\left\{  W_{k}\right\}
$ possesses nonempty upper limit in $X_{0}^{\alpha/2}$ such that
$\mathrm{Lim}\sup W_{k}\subset W_{0}.$ In the proof of that results we need
the following lemma in which the boundedness of the sequence $\left\{
W_{k}\right\}  $ is claimed.

\begin{lemma}
If the function $G$ satisfies conditions $(C1)-(C3)$, then for any boundary
data $v_{k}\in\mathcal{V}$ and for any parameter $u_{k}\in\mathcal{U}$ there
exists a ball $B_{\rho}=\left\{  w\in X_{0}^{\alpha/2}:\left\Vert w\right\Vert
_{X_{0}^{\alpha/2}}<\rho\right\}  \ $with $\rho>0$ such that $W_{k}\subset
B_{\rho}.$\label{lemat12}
\end{lemma}

\begin{proof}
First of all, let us observe that the set of values $\left\{  c_{k}:v_{k}%
\in\mathcal{V},\text{ }u_{k}\in\mathcal{U}\right\}  $ is bounded from above.
Indeed, for any $k\in\mathbb{N}$ and $g\left(  t\right)  =(1-t)\omega
_{0}+t\omega_{1}$ on $\left[  0,1\right]  ,$ conditions $(C2)$, $(C3)$ enable
us to infer that
\begin{align*}
c_{k}  &  =\inf\limits_{g\in M}\max\limits_{t\in\left[  0,1\right]
\ }\mathcal{F}_{k}\left(  g\left(  t\right)  \right)  \leq\max\limits_{t\in
\left[  0,1\right]  }\mathcal{F}_{k}\left(  (1-t)\omega_{0}+t\omega_{1}\right)
\\
&  \leq\max\limits_{t\in\left[  0,1\right]  }\left(  2c\left(  n,\alpha
\right)  \int\limits_{Q}\frac{(1-t)^{2}\left\vert \omega_{0}\left(  x\right)
-\omega_{0}\left(  y\right)  \right\vert ^{2}+t^{2}\left\vert \omega
_{1}\left(  x\right)  -\omega_{1}\left(  y\right)  \right\vert ^{2}%
}{\left\vert x-y\right\vert ^{n+\alpha}}dxdy\right. \\
&  +\left.  \frac{c\left(  n,\alpha\right)  }{2}\int\limits_{Q}\frac
{\left\vert v_{k}\left(  x\right)  -v_{k}\left(  y\right)  \right\vert ^{2}%
}{\left\vert x-y\right\vert ^{n+\alpha}}dxdy+\int\limits_{\Omega}G\left(
x,v_{k},u_{k}\right)  dx\right. \\
&  \left.  -\int\limits_{\Omega_{t}^{+}}G\left(  x,(1-t)\omega_{0}+t\omega
_{1}+v_{k},u_{k}\right)  dx-\int\limits_{\Omega_{t}^{-}}G\left(
x,(1-t)\omega_{0}+t\omega_{1}+v_{k},u_{k}\right)  dx\right) \\
&  \leq\max\limits_{t\in\left[  0,1\right]  }\left(  2c\left(  n,\alpha
\right)  (1-t)^{2}\left\Vert \omega_{0}\right\Vert _{X_{0}^{\alpha/2}}%
^{2}+2c\left(  n,\alpha\right)  t^{2}\left\Vert \omega_{1}\right\Vert
_{X_{0}^{\alpha/2}}^{2}-\frac{a}{p}\left\vert \Omega_{t}^{+}\right\vert
+c\left(  1+R^{s}\right)  \left\vert \Omega_{t}^{-}\right\vert \right) \\
&  +\frac{c\left(  n,\alpha\right)  }{2}\left\Vert v_{k}\right\Vert
_{Y^{\alpha/2}}^{2}+c\left\vert \Omega\right\vert +d\left\Vert v_{k}%
\right\Vert _{Y^{\alpha/2}}^{s}\\
&  \leq2c\left(  n,\alpha\right)  \max\left\{  \left\Vert \omega
_{0}\right\Vert _{X_{0}^{\alpha/2}}^{2},\left\Vert \omega_{1}\right\Vert
_{X_{0}^{\alpha/2}}^{2}\right\}  +D\leq\bar{c},
\end{align*}
and therefore
\[
c_{k}\leq\bar{c}%
\]
where $D,$ $c,$ $d,$ $\bar{c}$ are some constants, and the sets $\Omega
_{t}^{+},$ $\Omega_{t}^{-}$ are defined as
\[
\Omega_{t}^{+}=\left\{  x\in\Omega:\left\vert \left(  g\left(  t\right)
+v_{k}\right)  \left(  x\right)  \right\vert \geq R\right\}  ;\Omega_{t}%
^{-}=\left\{  x\in\Omega:\left\vert \left(  g\left(  t\right)  +v_{k}\right)
\left(  x\right)  \right\vert <R\right\}  .
\]
Then, for any $v_{k}\in\mathcal{V}$, $u_{k}\in\mathcal{U}$ and $w\in W_{k}$ we
have\ by $\left(  C2\right)  $ and $\left(  C3\right)  $%
\begin{align*}
p\bar{c}  &  \geq pc_{k}=p\mathcal{F}_{k}\left(  w\right)  -\left\langle
\mathcal{F}_{k}^{\prime}\left(  w\right)  ,w+v_{k}\right\rangle \\
&  =c\left(  n,\alpha\right)  \tfrac{p-2}{2}\left\Vert w\right\Vert
_{X_{0}^{\alpha/2}}^{2}+c\left(  n,\alpha\right)  \left(  p-2\right)
\int\limits_{Q}\frac{\left(  w(x\right)  -w(y))\left(  v_{k}\left(  x\right)
-v_{k}\left(  y\right)  \right)  }{\left\vert x-y\right\vert ^{n+\alpha}%
}dxdy\\
&  -c\left(  n,\alpha\right)  \int\limits_{Q}\frac{\left\vert v_{k}\left(
x\right)  -v_{k}\left(  y\right)  \right\vert ^{2}}{\left\vert x-y\right\vert
^{n+\alpha}}dxdy\\
&  +\int\limits_{\Omega}\left(  -pG\left(  x,w+v_{k},u_{k}\right)  +pG\left(
x,v_{k},u_{k}\right)  +G_{w}\left(  x,w+v_{k},u_{k}\right)  \left(
w+v_{k}\right)  \right)  dx\\
&  \geq c\left(  n,\alpha\right)  \left(  \frac{p-2}{2}\left\Vert w\right\Vert
_{X_{0}^{\alpha/2}}^{2}-\left(  p-2\right)  \left\Vert w\right\Vert
_{H_{0}^{\alpha/2}}\left[  v_{k}\right]  _{n,\alpha}-\left[  v_{k}\right]
_{n,\alpha}^{2}\right) \\
&  +\int\limits_{\Omega^{+}}\left(  -pG\left(  x,w+v_{k},u_{k}\right)
+pG\left(  x,v_{k},u_{k}\right)  +G_{w}\left(  x,w+v_{k},u_{k}\right)  \left(
w+v_{k}\right)  \right)  dx\\
&  +\int\limits_{\Omega^{-}}\left(  -pG\left(  x,w+v_{k},u_{k}\right)
+pG\left(  x,v_{k},u_{k}\right)  +G_{w}\left(  x,w+v_{k},u_{k}\right)  \left(
w+v_{k}\right)  \right)  dx\\
&  \geq c\left(  n,\alpha\right)  \left(  \frac{p-2}{2}\left\Vert w\right\Vert
_{X_{0}^{\alpha/2}}^{2}-\left(  p-2\right)  \left\Vert w\right\Vert
_{X_{0}^{\alpha/2}}\left[  v_{k}\right]  _{n,\alpha}-\left[  v_{k}\right]
_{n,\alpha}^{2}\right)  -pc\left\vert \Omega^{+}\right\vert -pc\left\Vert
v_{k}\right\Vert _{L^{s}\left(  \Omega\right)  }^{s}+d\\
&  \geq c\left(  n,\alpha\right)  \frac{p-2}{2}\left\Vert w\right\Vert
_{X_{0}^{\alpha/2}}^{2}-D_{1}\left\Vert w\right\Vert _{X_{0}^{\alpha/2}}-D_{2}%
\end{align*}
where $D_{1}$, $D_{2}$ are some positive constants and $\Omega^{+},$
$\Omega^{-}$ are some subsets of $\Omega$ of the form
\[
\Omega^{+}=\left\{  x\in\Omega:\left\vert \left(  w+v_{k}\right)  \left(
x\right)  \right\vert \geq R\right\}  ,\text{ }\Omega^{-}=\left\{  x\in
\Omega:\left\vert \left(  w+v_{k}\right)  \left(  x\right)  \right\vert
<R\right\}  .
\]
and $\left\langle \cdot,\cdot\right\rangle $ is a dual pair in $X_{0}%
^{\alpha/2}.$ Thus%
\begin{equation}
p\bar{c}\geq c\left(  n,\alpha\right)  \frac{p-2}{2}\left\Vert w\right\Vert
_{X_{0}^{\alpha/2}}^{2}+D_{1}\left\Vert w\right\Vert _{X_{0}^{\alpha/2}}%
+D_{2}. \label{3.6}%
\end{equation}
By condition $\left(  C3\right)  ,$ $p-2>0,$ and consequently there exists
$\rho>0$ such that $w\in B_{\rho}.$ Hence, $W_{k}\subset B_{\rho}$ for any
$v_{k}\in\mathcal{V}$ and $u_{k}\in\mathcal{U},$ which completes the proof.
\end{proof}

After necessary shift resulting in $\left(  \ref{problem2}\right)  $ and
$\mathcal{F}_{k}$ defined by $\left(  \ref{akcja2}\right)  $, we can assume,
without loss of generality, that from this point $\omega_{0}=0.$ It is
possible to demonstrate that there exist a bounded neighborhood $B$ of
$\omega_{0}$ in $X_{0}^{\alpha/2}$ and some point $\omega_{1}\notin
\overline{B}$ such that the assumptions of Theorem \ref{MPT} are fulfilled for
$\mathcal{F}_{k}$. Indeed, we have the following lemma.

\begin{lemma}
Suppose that conditions $\left(  C1\right)  -\left(  C4\right)  $ are
satisfied, the sequence $\left\{  v_{k}\right\}  \subset\mathcal{V}$ tends to
$v_{0}$ in $Y^{\alpha/2}$ and the sequence $\left\{  u_{k}\right\}
\subset\mathcal{U}$ tends to $u_{0}$ in $L^{\infty}$. Then for any
sufficiently large $k\in\mathbb{N}$, $v_{k}\in\mathcal{V}$ and $u_{k}%
\in\mathcal{U},$ there exist a ball $B_{\eta}\subset X_{0}^{\alpha/2}$ and an
element $\omega_{1}\notin\overline{B_{\eta}}$ such that $\,\inf\nolimits_{w\in
\partial B_{\eta}}\mathcal{F}_{k}\left(  w\right)  >0$ and $\mathcal{F}%
_{k}\left(  \omega_{1}\right)  <0,$ where $B_{\eta}=\left\{  w\in
X_{0}^{\alpha/2}:\left\Vert w\right\Vert _{Y^{\alpha/2}}<\eta\right\}  $ for
$\eta>0.\label{lemat13}$
\end{lemma}

\begin{proof}
From $\left(  C3\right)  ,$ in a similar fashion as in the proof of Lemma 4 in
\cite{SerVal}, there exists a constant $a_{0}>0$ such that
\[
G\left(  x,z,u\right)  \geq a_{0}\left\vert z\right\vert ^{p},
\]
for any $\left\vert z\right\vert \geq R,$ $u\in U$, a.e. $x\in\Omega$ and
$R>0$ as in $\left(  C3\right)  $ with $p>2$. On combining the above
inequality with condition $\left(  C2\right)  $ we deduce that there exists a
positive constant $a_{1}$ such that
\begin{equation}
G\left(  x,z,u\right)  \geq a_{0}\left\vert z\right\vert ^{p}-a_{1}
\label{(*)}%
\end{equation}
for $z\in\mathbb{R},$ $u\in U$ and a.e. $x\in\Omega$, with $p\in\left(
2,2_{\alpha}^{\ast}\right)  $. Moreover, by $\left(  C2\right)  $, $\left(
C4\right)  $ in the same way as in the proof of Lemma 3 in \cite{SerVal2}
there exist $b\in\left(  0,\frac{c\left(  n,\alpha\right)  }{2}\right)  $\ and
$a_{2}>0$\ such that
\begin{equation}
\left\vert G\left(  x,z,u\right)  +\frac{c\left(  n,\alpha\right)  }{2}%
z^{2}\right\vert \leq\frac{b}{2}\left\vert z-v_{0}\left(  x\right)
\right\vert ^{2}+a_{2}\left\vert z-v_{0}\left(  x\right)  \right\vert ^{s}
\label{(**)}%
\end{equation}
\ for any $z\in\mathbb{R}$, $u\in U,$ a.e.\ $x\in\Omega$\ and\textbf{ }%
$s\in\left(  2,2_{\alpha}^{\ast}\right)  .$ \newline For fixed $k\in
\mathbb{N},$ the orthogonality condition $\left(  w,v_{k}\right)
_{Y^{\alpha/2}}=0$ and inequality in $\left(  \ref{(**)}\right)  $ lead to the
following estimate%
\begin{align*}
\mathcal{F}_{k}\left(  w\right)   &  =c\left(  n,\alpha\right)  \left(
\int\limits_{Q}\frac{\left\vert w\left(  x\right)  -w\left(  y\right)
\right\vert ^{2}}{2\left\vert x-y\right\vert ^{n+\alpha}}dxdy+\int
\limits_{Q}\frac{\left(  v_{k}\left(  x\right)  -v_{k}\left(  y\right)
\right)  \left(  w\left(  x\right)  -w\left(  y\right)  \right)  }{\left\vert
x-y\right\vert ^{n+\alpha}}dxdy\right) \\
&  -\int\limits_{\Omega}\left(  G\left(  x,w+v_{k},u_{k}\right)  -G\left(
x,v_{k},u_{k}\right)  \right)  dx\\
&  =\tfrac{c\left(  n,\alpha\right)  }{2}\left\Vert w\right\Vert
_{Y^{\alpha/2}}^{2}+c\left(  n,\alpha\right)  \left(  w,v_{k}\right)
_{Y^{\alpha/2}}\\
&  -\int\limits_{\Omega}\left(  G\left(  x,w+v_{k},u_{k}\right)
+\tfrac{c\left(  n,\alpha\right)  }{2}\left(  w+v_{k}\right)  ^{2}\right)
dx+\int\limits_{\Omega}\left(  G\left(  x,v_{k},u_{k}\right)  +\tfrac{c\left(
n,\alpha\right)  }{2}v_{k}^{2}\right)  dx\\
&  \geq\left(  \tfrac{c\left(  n,\alpha\right)  }{2}-b\right)  \left\Vert
w\right\Vert _{Y^{\alpha/2}}^{2}-C_{1}\left\Vert w\right\Vert _{Y^{\alpha/2}%
}^{s}-C_{2}\left\Vert v_{k}-v_{0}\right\Vert _{Y^{\alpha/2}}^{2}%
-C_{3}\left\Vert v_{k}-v_{0}\right\Vert _{Y^{\alpha/2}}^{s}%
\end{align*}
where constants $C_{1},$ $C_{2},$ $C_{3}$ are positive. Since, by $\left(
C4\right)  ,$ $b<\frac{c\left(  n,\alpha\right)  }{2}$ and $v_{k}\rightarrow
v_{0}$ in $Y^{\alpha/2}$ while $s>2,$ it follows that there exists a constant
$\eta>0$ such that $\inf\nolimits_{w\in\partial B_{\eta}}\mathcal{F}%
_{k}\left(  w\right)  \geq\varepsilon>0$ for any $k$ sufficiently large$.$

To finish the proof it is enough to show that for any $v_{k}\in\mathcal{V}$
and $u_{k}\in\mathcal{U}$ there exists $\omega_{1}\notin\overline{B_{\eta}}$
such that $\mathcal{F}_{k}\left(  \omega_{1}\right)  <0.$ For a fixed nonzero
$w\in X_{0}^{\alpha/2}$ and $l>0,$ from $\left(  C2\right)  $ and $\left(
C3\right)  $ hence $\left(  \ref{(*)}\right)  $\ we have the following
estimates
\begin{align*}
\mathcal{F}_{k}\left(  lw\right)   &  =c\left(  n,\alpha\right)  \left(
\int\limits_{Q}\frac{l^{2}\left\vert w\left(  x\right)  -w\left(  y\right)
\right\vert ^{2}}{2\left\vert x-y\right\vert ^{n+\alpha}}dxdy+\int
\limits_{Q}\frac{l\left(  v_{k}\left(  x\right)  -v_{k}\left(  y\right)
\right)  \left(  w\left(  x\right)  -w\left(  y\right)  \right)  }{\left\vert
x-y\right\vert ^{n+\alpha}}dxdy\right) \\
&  -\int\limits_{\Omega}\left(  G\left(  x,lw+v_{k},u_{k}\right)  -G\left(
x,v_{k},u_{k}\right)  \right)  dx\\
&  \leq\tfrac{c\left(  n,\alpha\right)  }{2}l^{2}\left\Vert w\right\Vert
_{Y^{\alpha/2}}^{2}+c\left(  n,\alpha\right)  lC_{4}\left\Vert w\right\Vert
_{Y^{\alpha/2}}-\int\limits_{\Omega}\left(  a_{0}\left\vert lw+v_{k}%
\right\vert ^{p}-a_{1}\right)  dx+\int\limits_{\Omega}G\left(  x,v_{k}%
,u_{k}\right)  dx\\
&  \leq\tfrac{c\left(  n,\alpha\right)  }{2}l^{2}\left\Vert w\right\Vert
_{Y^{\alpha/2}}^{2}+c\left(  n,\alpha\right)  lC_{4}\left\Vert w\right\Vert
_{Y^{\alpha/2}}-a_{0}l^{p}\int\limits_{\Omega}\left(  \left\vert
w+\tfrac{v_{k}}{l}\right\vert ^{p}\right)  dx+C_{5}%
\end{align*}
where $C_{4},C_{5},a_{0}>0.$ Since $p\in\left(  2,2_{\alpha}^{\ast}\right)  $
and $a_{0}>0$ from $\left(  \ref{(*)}\right)  $ we get that $\lim
\nolimits_{l\rightarrow\infty}\mathcal{F}_{k}\left(  lw\right)  =-\infty.$
Therefore, there exists $l_{0}>0$ such that for $\omega_{1}=l_{0}w$ we have
$\left\Vert \omega_{1}\right\Vert _{Y^{\alpha/2}}\geq\eta$ and $\mathcal{F}%
_{k}\left(  \omega_{1}\right)  <0\,$\ for any $v_{k}\in\mathcal{V}$ and
$u_{k}\in\mathcal{U},$ which proves the assertion of the lemma.
\end{proof}

Lemmas \ref{lemat12} and \ref{lemat13} give that the geometry of Mountain Pass
Theorem is fulfilled by $\mathcal{F}_{k}.$ Therefore, in order to apply such
Mountain Pass Theorem, we are left with checking the validity of the
Palais-Smale condition. This will be accomplished in the forthcoming lemma.

\begin{lemma}
Suppose that conditions $\left(  C1\right)  -\left(  C4\right)  $ are
satisfied. Then for any $k\in\mathbb{N}_{0},$ $\mathcal{F}_{k}$ satisfies the
Palais-Smale condition.\label{Lemat_PS}
\end{lemma}

\begin{proof}
Let $k$ be fixed and $\left\{  w_{i}\right\}  $ be a Palais-Smale sequence so
that $\left\{  \mathcal{F}_{k}\left(  w_{i}\right)  \right\}  $ is bounded and
$\mathcal{F}_{k}^{\prime}\left(  w_{i}\right)  \rightarrow0$ as $i\rightarrow
\infty.$ Thus, there exist constants $C_{1},$ $C_{2}>0$ such that $\left\vert
\mathcal{F}_{k}\left(  w_{i}\right)  \right\vert \leq C_{1}$ and $\left\Vert
\mathcal{F}_{k}^{\prime}\left(  w_{i}\right)  \right\Vert \leq C_{2}$ for all
$i\in\mathbb{N}$. In the same manner as in the proof of Lemma \ref{lemat12},
we obtain the following estimates%
\begin{align*}
pC_{1}+C_{2}\left\Vert v_{k}\right\Vert _{Y^{\alpha/2}}+C_{3}\left\Vert
w_{i}\right\Vert _{X_{0}^{\alpha/2}}  &  \geq pC_{1}+C_{2}\left\Vert
w_{i}+v_{k}\right\Vert _{Y^{\alpha/2}}\\
&  \geq p\mathcal{F}_{k}\left(  w_{i}\right)  -\left\langle \mathcal{F}%
_{k}^{\prime}\left(  w_{i}\right)  ,w_{i}+v_{k}\right\rangle \\
&  \geq c\left(  n,\alpha\right)  \tfrac{p-2}{2}\left\Vert w_{i}\right\Vert
_{X_{0}^{\alpha/2}}^{2}-D_{1}\left\Vert w_{i}\right\Vert _{X_{0}^{\alpha/2}%
}-D_{2},
\end{align*}
where $D_{1}$, $D_{2}$ are some positive constants from the proof of Lemma
\ref{lemat12} and $C_{3}>0$. Hence
\[
\left\Vert w_{i}\right\Vert _{X_{0}^{\alpha/2}}^{2}\leq\tfrac{2}{c\left(
n,\alpha\right)  \left(  p-2\right)  }\left(  C_{4}\left\Vert w_{i}\right\Vert
_{X_{0}^{\alpha/2}}+pC_{1}+C_{2}\left\Vert v_{k}\right\Vert _{Y^{\alpha/2}%
}+D_{2}\right)  \text{ for }i\in\mathbb{N}\text{,}%
\]
where $C_{3}>0$. Therefore, the sequence $\left\{  w_{i}\right\}  $ is bounded
in $X_{0}^{\alpha/2}$ and as such it contains a subsequence, still denoted by
$\left\{  w_{i}\right\}  ,$ such that $w_{i}$ tends to $w_{0}$ weakly in
$X_{0}^{\alpha/2},$ since the space $X_{0}^{\alpha/2}$ is reflexive. From the
fact that the space $X_{0}^{\alpha/2}$ is compactly embedded into the space
$L^{s}\left(  \Omega\right)  $ with $s\in\left[  1,2_{\alpha}^{\ast}\right)
,$ we may assume after passing to a subsequence, still labelled $\left\{
w_{i}\right\}  $, that $w_{i}\rightarrow w_{0}$ in $L^{s}\left(
\Omega\right)  .$ Consequently, denoting $\left\langle \cdot,\cdot
\right\rangle $ a dual pair in $X_{0}^{\alpha/2},$ we have
\[
\left\langle \mathcal{F}_{k}^{\prime}\left(  w_{i}\right)  -\mathcal{F}%
_{k}^{\prime}\left(  w_{0}\right)  ,w_{i}-w_{0}\right\rangle \underset
{i\rightarrow\infty}{\rightarrow}0.
\]
The equality
\begin{align*}
&  \left\langle \mathcal{F}_{k}^{\prime}\left(  w_{i}\right)  -\mathcal{F}%
_{k}^{\prime}\left(  w_{0}\right)  ,w_{i}-w_{0}\right\rangle \\
&  =c\left(  n,\alpha\right)  \left\Vert w_{i}-w_{0}\right\Vert _{X_{0}%
^{\alpha/2}}^{2}+\int\limits_{\Omega}\left(  G_{w}\left(  x,w_{0}+v_{k}%
,u_{k}\right)  -G_{w}\left(  x,w_{i}+v_{k},u_{k}\right)  \right)  \left(
w_{i}-w_{0}\right)  dx
\end{align*}
and the growth condition $\left(  C2\right)  $ lead by H\"{o}lder inequality
to
\begin{align*}
&  \left\vert \int\limits_{\Omega}\left(  G_{w}\left(  x,w_{0}+v_{k}%
,u_{k}\right)  -G_{w}\left(  x,w_{i}+v_{k},u_{k}\right)  \right)  \left(
w_{i}-w_{0}\right)  dx\right\vert \\
&  \leq\left\Vert w_{i}-w_{0}\right\Vert _{L^{s}}\left(  \int\limits_{\Omega
}\left\vert G_{w}\left(  x,w_{i}+v_{k},u_{k}\right)  -G_{w}\left(
x,w_{0}+v_{k},u_{k}\right)  \right\vert ^{\frac{s}{s-1}}dx\right)
^{\frac{s-1}{s}}.
\end{align*}
Let us notice that the right hand side of the above inequality tends as
$i\rightarrow\infty$ to 0, by the growth condition $\left(  C2\right)  $ and
since $w_{i}\rightarrow w_{0}$ in $L^{s}\left(  \Omega\right)  $. As a result,
$w_{i}\rightarrow w_{0}$ in $X_{0}^{\alpha/2}.$ In that way we have
demonstrated that for any $k\in\mathbb{N}_{0}$ the functional $\mathcal{F}%
_{k}$ satisfies the Palais-Smale condition guaranteeing the required
compactness property.
\end{proof}

\section{The main continuity result \label{mainresult}}

In this section we state and prove some sufficient conditions under which
critical points of mountain pass type of the functional $\mathcal{F}_{k}$
defined in $\left(  \ref{akcja2}\right)  $ exist and depend continuously on
parameters. First of all, we state some sufficient conditions guaranteeing a
uniform convergence on any ball in the space $X_{0}^{\alpha/2}$ of a sequence
of functionals together with a sequence of their derivatives. This is in fact
verification of the assumption $\left(  b\right)  $ of Lemma \ref{lemat11}.

\begin{lemma}
If conditions $\left(  C1\right)  $, $\left(  C2\right)  ,$ $\left(
C5\right)  $ are satisfied, the sequence $\left\{  v_{k}\right\}
\subset\mathcal{V}\,$\ tends to $v_{0}$ in $Y^{\alpha/2}$ while the sequence
$\left\{  u_{k}\right\}  \subset\mathcal{U}$ tends to $u_{0}$ in $L^{\infty},$
then the sequences $\left\{  \mathcal{F}_{k}\right\}  ,$ $\left\{
\mathcal{F}_{k}^{\prime}\right\}  $ tend uniformly on any ball from
$X_{0}^{\alpha/2}$ to $\mathcal{F}_{0}$ and $\mathcal{F}_{0}^{\prime},$
respectively.$\label{lemat14}$
\end{lemma}

\begin{proof}
For any $B_{\rho}\subset X_{0}^{\alpha/2}$ and $w\in B_{\rho},$ we have
\begin{align*}
&  \left\vert \mathcal{F}_{k}\left(  w\right)  -\mathcal{F}_{0}\left(
w\right)  \right\vert \\
&  \leq c\left(  n,\alpha\right)  \int\limits_{Q}\frac{\left(  \left(
v_{k}-v_{0}\right)  \left(  x\right)  -\left(  v_{k}-v_{0}\right)  \left(
y\right)  \right)  \left(  w\left(  x\right)  -w\left(  y\right)  \right)
}{\left\vert x-y\right\vert ^{n+\alpha}}dxdy\\
&  +\left\vert \int\limits_{\Omega}\left(  G\left(  x,v_{k},u_{k}\right)
-G\left(  x,v_{0},u_{0}\right)  \right)  dx\right.  \left.  +\int
\limits_{\Omega}\left(  G\left(  x,w+v_{k},u_{k}\right)  -G\left(
x,w+v_{0},u_{0}\right)  \right)  dx\right\vert \\
&  \leq c\left(  n,\alpha\right)  \rho\left\Vert v_{k}-v_{0}\right\Vert
_{Y^{\alpha/2}}+\int\limits_{\Omega}\left\vert G\left(  x,v_{k},u_{0}\right)
-G\left(  x,v_{0},u_{0}\right)  \right\vert dx\\
&  +\int\limits_{\Omega}\left\vert G\left(  x,w+v_{k},u_{0}\right)  -G\left(
x,w+v_{0},u_{0}\right)  \right\vert dx+\left\Vert u_{k}-u_{0}\right\Vert
_{L^{\infty}}\left(  D_{1}+D_{2}\left\Vert v_{k}\right\Vert _{Y^{\alpha/2}%
}^{2}\right)  <\varepsilon
\end{align*}
for any fixed $\varepsilon>0$ and sufficiently large $k$. Indeed, let us
observe that $v_{k}$ tends to $v_{0}$ in $Y^{\alpha/2}$. By conditions
$\left(  C1\right)  ,$ $\left(  C2\right)  $, $\left(  C5\right)  $, the
Krasnosielskii Theorem on continuity of Niemyckii's operator, cf.
\cite{IdcRog}, implies that the right hand side of the above inequality tends
to $0$ for any $w\in B_{\rho}.$ It means that the sequence $\left\{
\mathcal{F}_{k}\right\}  $ tends uniformly to $\mathcal{F}_{0}$ on a ball
$B_{\rho}.$\newline A similar reasoning holds for the case of a uniform
convergence of the sequence $\left\{  \mathcal{F}_{k}^{\prime}\right\}  $ to
$\mathcal{F}_{0}^{\prime}$ on a ball from $X_{0}^{\alpha/2}.$ Let us take any
ball $B_{\rho}\subset X_{0}^{\alpha/2}.$ For any $w\in B_{\rho}$ and $h\in
X_{0}^{\alpha/2}$ such that $h\in B_{1}$ simple calculations lead to
\begin{align*}
&  \left\vert \left\langle \mathcal{F}_{k}^{\prime}\left(  w\right)
-\mathcal{F}_{0}^{\prime}\left(  w\right)  ,h\right\rangle \right\vert \\
&  \leq c\left(  n,\alpha\right)  \left\Vert v_{k}-v_{0}\right\Vert
_{Y^{\alpha/2}}+\int\limits_{\Omega}\left\vert G_{w}\left(  x,w+v_{k}%
,u_{k}\right)  h-G_{w}\left(  x,w+v_{0},u_{0}\right)  h\right\vert dx\\
&  \leq c\left(  n,\alpha\right)  \left\Vert v_{k}-v_{0}\right\Vert
_{Y^{\alpha/2}}+\left(  \int\limits_{\Omega}\left\vert G_{w}\left(
x,w+v_{k},u_{0}\right)  -G_{w}\left(  x,w+v_{0},u_{0}\right)  \right\vert
^{\frac{s}{s-1}}dx\right)  ^{\frac{s-1}{s}}\\
&  +\left\Vert u_{k}-u_{0}\right\Vert _{L^{\infty}}\left(  D_{3}%
+D_{4}\left\Vert v_{k}\right\Vert _{Y^{\alpha/2}}\right)  <\varepsilon
\end{align*}
for sufficiently large $k$, and the claim of the lemma follows.
\end{proof}

Now we employ Lemmas \ref{lemat11}, \ref{lemat12}, \ref{lemat13},
\ref{Lemat_PS}, \ref{lemat14} and Theorem \ref{MPT} to prove the following theorem:

\begin{theorem}
Suppose that the function $G$ satisfies conditions $\left(  C1\right)
-\left(  C5\right)  $ and moreover the sequence $\left\{  v_{k}\right\}
\subset\mathcal{V}$ tends to $v_{0}$ in $Y^{\alpha/2}$ and the sequence
$\left\{  u_{k}\right\}  \subset\mathcal{U}$ tends to $u_{0}$ in $L^{\infty}$.
Then \newline(i) for any $k$ large enough$,$ the set of critical points
$W_{k}$ of the functional $\mathcal{F}_{k}$ is nonempty and does not contain
zero and in fact any point with zero value,\newline(ii) any sequence $\left\{
w_{k}\right\}  $ such that $w_{k}\in W_{k},$ $k\in\mathbb{N},$ is relatively
compact in $X_{0}^{\alpha/2}$ and $\mathrm{Lim}\sup W_{k}\subset
W_{0}.\label{twierdzenie11}$
\end{theorem}

\begin{proof}
First, we shall prove the first part of the assertion of the theorem, that is
for any $k$ sufficiently large$,$ the set of critical points $W_{k}$ of the
functional $\mathcal{F}_{k}$ is nonempty and does not contain zero and any
point with zero value. Obviously, the functional $\mathcal{F}_{k},$
$k\in\mathbb{N}_{0}$ is of $C^{1}-$class on $X_{0}^{\alpha/2}$. Moreover, from
Lemma \ref{lemat13} it follows that there exist the ball $B_{\eta}$ and the
point $\omega_{1}\in X_{0}^{\alpha/2},$ independent of the choice of
$v_{k},u_{k}$, such that $\omega_{1}\notin\overline{B_{\eta}}$ and
$\inf\nolimits_{w\in\partial B_{\eta}}\mathcal{F}_{k}\left(  w\right)
>0=\max\left\{  \mathcal{F}_{k}\left(  0\right)  ,\mathcal{F}_{k}\left(
\omega_{1}\right)  \right\}  ,$ $k\in\mathbb{N}_{0}$ and furthermore
conditions $\left(  C2\right)  $, $\left(  C3\right)  $ guarantee that the
functional $\mathcal{F}_{k}$ satisfies the Palais-Smale condition for
$k\in\mathbb{N}_{0}$ as stated in Lemma \ref{Lemat_PS}. At this point we apply
the Mountain Pass Theorem \ref{MPT}, with $\omega_{0}=0$ and $c=c_{k},$ to
deduce that for any $v_{k}$ and $u_{k},$ the set of critical points for which
a critical value of the functional $\mathcal{F}_{k}$ denoted by $c_{k}$ is
attained, is not empty, i.e.%
\[
W_{k}=\left\{  w\in X_{0}^{\alpha/2}:\mathcal{F}_{k}\left(  w\right)
=c_{k}\text{ and }\mathcal{F}_{k}^{\prime}\left(  w\right)  =0\right\}
\neq\emptyset.
\]
Moreover, $c_{k}=\inf\nolimits_{g\in M}\max\nolimits_{t\in\left[  0,1\right]
}\mathcal{F}_{k}\left(  g\left(  t\right)  \right)  >\max\left\{
\mathcal{F}_{k}\left(  0\right)  ,\mathcal{F}_{k}\left(  \omega_{1}\right)
\right\}  =0,$ and therefore $w=0$ does not belong to the set $W_{k}$ for all
$k\in\mathbb{N}_{0}$ and actually any point with zero value$.$\newline What is
left is to demonstrate that $\mathrm{Lim}\sup W_{k}\neq\emptyset$ in
$X_{0}^{\alpha/2}$ and $\mathrm{Lim}\sup W_{k}\subset W_{0}.$ This part of the
assertion of our theorem follows directly from Lemma \ref{lemat11}. Indeed, by
invoking Lemma \ref{lemat12}, we obtain that there exists a ball $B_{\rho
}\subset X_{0}^{\alpha/2}$ with $\rho>0$ such that $W_{k}\subset B_{\rho}$ for
all $k\in\mathbb{N}_{0},$ and subsequently there exists a ball $B_{r}\subset
X_{0}^{\alpha/2}$ with $r\geq\rho$ such that $w_{1}\in B_{r},$ i.e.
$W_{k}\left(  r\right)  =W_{k},$ where $W_{k}\left(  r\right)  $ is given by
$\left(  \ref{(1.13)}\right)  $ with $\mathcal{I}_{k}=\mathcal{F}_{k},$
$k\in\mathbb{N}_{0}.$ $\mathcal{F}_{0}$ satisfies the Palais-Smale condition
and each functional $\mathcal{F}_{k}$ is of $C^{1}-$class, and corresponding
set $W_{k}\left(  r\right)  =W_{k}$ is nonempty for $k\in\mathbb{N}_{0}$.
Finally, Lemma \ref{lemat14} implies that the sequences $\left\{
\mathcal{F}_{k}\right\}  ,$ $\left\{  \mathcal{F}_{k}^{\prime}\right\}  $ tend
uniformly on any ball from $X_{0}^{\alpha/2}$ to $\mathcal{F}_{0}$ and
$\mathcal{F}_{0}^{\prime},$ respectively$.$ It means that the proof of the
theorem is complete.
\end{proof}

Let us notice that the critical value of the functional $F_{u_{k},v_{k}}$
defined in $\left(  \ref{akcja1}\right)  $ denoted by $\mathbf{c}_{k}$
satisfies the following relation
\[
\mathbf{c}_{k}=c_{k}-\int\limits_{\Omega}G\left(  x,v_{k}\left(  x\right)
,u_{k}\left(  x\right)  \right)  dx
\]
where $c_{k}$ is defined in $\left(  \ref{(1.10)}\right)  $ as the critical
value of $\mathcal{F}_{k}.$ The set of critical points of the functional
$F_{u_{k},v_{k}}$ for which the critical value $\mathbf{c}_{k}$ is attained
has the form%
\[
W_{u_{k},v_{k}}=\left\{  w\in X_{0}^{\alpha/2}:F_{u_{k},v_{k}}\left(
w\right)  =\mathbf{c}_{k}\text{ and }F_{u_{k},v_{k}}^{\prime}\left(  w\right)
=0\right\}
\]
for $k\in\mathbb{N}_{0}.$

Immediately from Theorem \ref{twierdzenie11} we get the following corollaries
characterizing, first of all, the set of critical points of mountain pass type
of the functional of action without shift and then the set of corresponding
weak solutions of the problem involving the equation with the fractional
Laplacian with exterior homogenous Dirichlet boundary data.

\begin{corollary}
If all assumptions of Theorem \ref{twierdzenie11} are satisfied, then for any
$k$ sufficiently large the set $W_{u_{k},v_{k}}$ is nonempty and does not
contain zero and any point with zero value, $\mathrm{Lim}\sup W_{u_{k},v_{k}%
}\neq\emptyset$ in $X_{0}^{\alpha/2}$ and $\mathrm{Lim}\sup W_{u_{k},v_{k}%
}\subset W_{u_{0},v_{0}}.$\label{Cor4.1}
\end{corollary}

In other words, the set valued mapping%
\[
L^{\infty}\times Y^{\alpha/2}\ni\left(  u,v\right)  \mapsto W_{u,v}\subset
X_{0}^{\alpha/2}%
\]
is upper semicontinous with respect to the norm topologies of spaces
$L^{\infty},$ $Y^{\alpha/2}$ and $X_{0}^{\alpha/2}.$

Let us denote by $\mathcal{S}_{u_{k},v_{k}}^{w}$ the set of the weak solutions
to problem $\left(  \ref{2.1}\right)  $ defined in $\left(  \ref{weak}\right)
$ corresponding to the critical value $\mathbf{c}_{k}.$

\begin{corollary}
If all assumptions of Theorem \ref{twierdzenie11} are satisfied, then for any
$k$ sufficiently large the set $\mathcal{S}_{u_{k},v_{k}}^{w}$ is nonempty,
$\mathrm{Lim}\sup\mathcal{S}_{u_{k},v_{k}}^{w}\neq\emptyset$ in $X_{0}%
^{\alpha/2}$ and $\mathrm{Lim}\sup\mathcal{S}_{u_{k},v_{k}}^{w}\subset
\mathcal{S}_{u_{0},v_{0}}^{w}.$
\end{corollary}

Furthermore, it is easy to observe that $\mathcal{S}_{u_{k},v_{k}}%
^{z}=\mathcal{S}_{u_{k},v_{k}}^{w}+v_{k},$\ $k\in N_{0}$ is a set of weak
solutions to problem\textbf{ }$\left(  \ref{1.1}\right)  $ with nonhomogenous
exterior boundary data corresponding to the critical value $\mathbf{c}_{k}$.

\begin{corollary}
If all assumptions of Theorem \ref{twierdzenie11} are satisfied, then for any
$k$ the set $\mathcal{S}_{u_{k},v_{k}}^{z}$ is nonempty and does not contain
$v_{k}$, $\mathrm{Lim}\sup\mathcal{S}_{u_{k},v_{k}}^{z}\neq\emptyset$ in
$Y^{\alpha/2}$ and $\mathrm{Lim}\sup\mathcal{S}_{u_{k},v_{k}}^{z}%
\subset\mathcal{S}_{u_{0},v_{0}}^{z}.\label{Cor4.2}$
\end{corollary}

\begin{example}
\label{przyklad}It is easy to check that the assumptions of Theorem
\ref{twierdzenie11} are satisfied by the equation
\begin{equation}
\left\{
\begin{array}
[c]{l}%
\left(  -\triangle\right)  ^{3/4}z\left(  x\right)  =\frac{7}{2}z^{5/2}\left(
x\right)  -\gamma u\left(  x\right)  z\left(  x\right)  -\frac{5}{2}u\left(
x\right)  z^{3/2}\left(  x\right)  \mathrm{sin}^{2}\left\vert x\right\vert
\text{ \textrm{in} }\Omega\subset\mathbb{R}^{3}\\
z\left(  x\right)  =v\left(  x\right)  \text{ \textrm{in} }\mathbb{R}%
^{3}\backslash\Omega,
\end{array}
\right.  \label{4.1}%
\end{equation}
where $\Omega=(0,1)^{3},$ $u\in\mathcal{U}$ such that
\[
\mathcal{U}=\left\{  u\in L^{\infty}\left(  \Omega,\mathbb{R}\right)
:\right.  \left.  u\left(  x\right)  \in U\subset\left(  \frac{\hat{c}%
-b}{\gamma},\frac{\hat{c}+b}{\gamma}\right)  \text{ a.e.}\right\}
\]
where $U$ is a bounded interval $\gamma>0,$ $\hat{c}=c\left(  3,\frac{3}%
{2}\right)  =\frac{3\Gamma\left(  \frac{9}{4}\right)  }{\left(  2\pi\right)
^{3/2}\Gamma\left(  \frac{1}{4}\right)  }=\frac{15}{2^{11/2}\pi^{3/2}}$ and
$b$ is like in $\left(  C4\right)  $ while $v\in\mathcal{V}$ such that
\[
\mathcal{V=}\left\{  v\in Y^{3/4}:v\left(  x\right)  \in\left[  0,1\right]
,\text{ }\left\Vert v\right\Vert \leq1\right\}  .
\]
Let us notice that the functional of action for the equation with homogenous
exterior boundary condition related to $F_{u,v}$ defined in $\left(
\ref{akcja1}\right)  ,$ and to $\mathcal{F}_{k}$ defined in $\left(
\ref{akcja2}\right)  $ has the following form%
\begin{equation}
F\left(  z\right)  =\frac{\hat{c}}{2}\int\limits_{Q}\frac{\left\vert z\left(
x\right)  -z\left(  y\right)  \right\vert ^{2}}{\left\vert x-y\right\vert
^{\frac{9}{2}}}dxdy-\int\limits_{\Omega}\left(  z^{\frac{7}{2}}\left(
x\right)  -\tfrac{\gamma}{2}u\left(  x\right)  z^{2}\left(  x\right)
-u\left(  x\right)  z^{\frac{5}{2}}\left(  x\right)  \mathrm{sin}%
^{2}\left\vert x\right\vert \right)  dx \label{4.2}%
\end{equation}
where $z\in X_{0}^{3/4}.$ It is easy to check that all assumption $\left(
C1\right)  -\left(  C5\right)  $ are satisfied by the function
\[
G\left(  x,z\left(  x\right)  ,u\left(  x\right)  \right)  =z^{7/2}\left(
x\right)  -\tfrac{\gamma}{2}u\left(  x\right)  z^{2}\left(  x\right)
-u\left(  x\right)  z^{5/2}\left(  x\right)  \mathrm{sin}^{2}\left\vert
x\right\vert .
\]
Applying Corollary \ref{Cor4.2} we have that for any $u\in\mathcal{U}$ and
$v\in\mathcal{V}$ there exists a weak solution $z_{u,v}\in Y^{3/4}$ to problem
$\left(  \ref{4.1}\right)  $ and the set of weak solutions continuously, or
upper semicontinuously, depends on boundary data $v\rightarrow0$ and
distributed parameter $u.$
\end{example}

\begin{remark}
\label{uwaga}Let us observe that putting, $\widetilde{z}_{k}\left(  x\right)
=k\rho_{1}\left(  x\right)  $ where $\rho_{1}$ is the positive first
eigenfunction of the fractional Laplace operator defined in $\left(
\ref{Laplacian_i}\right)  $ with $\alpha=\frac{3}{2}$ we get for some
$\delta>0$
\begin{align*}
F\left(  \widetilde{z}_{k}\right)   &  \leq\frac{\hat{c}}{2}k^{2}\left\Vert
\rho_{1}\right\Vert _{X_{0}^{\alpha/2}}^{2}-k^{7/2}\left\Vert \rho
_{1}\right\Vert _{L^{7/2}}^{7/2}-\frac{\gamma}{2}\left\Vert u\right\Vert
_{L^{2}}k^{2}\left\Vert \rho_{1}\right\Vert _{L^{4}}^{2}+k^{5/2}\left\Vert
u\right\Vert _{L^{\infty}}\left\Vert \rho_{1}\right\Vert _{L^{5/2}}^{5/2}\\
&  \leq-k^{7/2}\left(  \left\Vert \rho_{1}\right\Vert _{L^{7/2}}%
^{7/2}+k^{-\delta}A\right)  \rightarrow-\infty
\end{align*}
as $k\rightarrow\infty$ for $F$ is defined in $\left(  \ref{4.2}\right)  .$
Moreover, for a sequence $\overline{z}_{k}\left(  x\right)  =\rho_{k}\left(
x\right)  $ such that
\[
\left(  -\Delta\right)  ^{3/4}\rho_{k}=\lambda_{k}\rho_{k}%
\]
$\left\Vert \rho_{k}\right\Vert _{L^{2}}=1$ and $\left\Vert \rho
_{k}\right\Vert _{L^{\infty}}\leq\tilde{C}$ for some positive constant
$\tilde{C}$ we obtain%
\[
F\left(  \overline{z}_{k}\right)  \geq\frac{\hat{c}}{2}\lambda_{k}\left\Vert
\rho_{k}\right\Vert _{L^{2}}^{2}-C=\frac{\hat{c}}{2}\lambda_{k}-C\rightarrow
\infty
\]
as $k\rightarrow\infty$ and $\lambda_{k}\sim k^{1/2\text{ }}$(cf. for example
\cite{CheSon}) while $C$ depends on $\left\Vert u\right\Vert _{L^{\infty}}$
and $\left\Vert \rho_{k}\right\Vert _{L^{\infty}}$ in a bounded way.
Consequently, the functional $F$ given in $\left(  \ref{4.2}\right)  $ is
unbounded both from above and below. For that reason to obtain the existence
results we cannot use methods applied, for example, in \cite{Bor3}.
\end{remark}

\section{Optimal control problem}

Following ideas employed for optimal control system in \cite{MacStr} by a
direct application of Corollary \ref{Cor4.2} following from Theorem
\ref{twierdzenie11} the existence of optimal processes can be ascertained for
the optimal control problem involving the weak formulation of the differential
equation with\ the fractional Laplacian and some nonlinearity of the form
\begin{equation}
(-\Delta)^{\alpha/2}z\left(  x\right)  =G_{z}\left(  x,z\left(  x\right)
,u\left(  x\right)  \right)  \text{ in }\Omega\label{5.1}%
\end{equation}
with the fixed exterior boundary condition
\begin{equation}
z\left(  x\right)  =v\left(  x\right)  \text{ on }\mathbb{R}^{n}%
\backslash\Omega\label{5.1b}%
\end{equation}
and with the integral cost functional $J$ defined via $\Phi$ as
\begin{equation}
J\left(  z,u\right)  =\int\limits_{\Omega}\Phi\left(  x,z\left(  x\right)
,u\left(  x\right)  \right)  dx \label{5.2}%
\end{equation}
where $\Phi:\Omega\times\mathbb{R\times R}^{m}\rightarrow\mathbb{R}$ is a
given function and a control $u\in\mathcal{U}_{\lambda}$ where%
\[
\mathcal{U}_{\lambda}=\left\{  u:\Omega\rightarrow\mathbb{R}^{m}:u\left(
x\right)  \in U\text{ and }\left\vert u\left(  x_{1}\right)  -u\left(
x_{2}\right)  \right\vert <\lambda\text{ }\left\vert x_{1}-x_{2}\right\vert
\right\}
\]
for a fixed $\lambda>0$ and $U$ a compact subset of $\mathbb{R}^{m}.$ Recall
that $\alpha\in\left(  0,2\right)  ,$ $n>\alpha$ and $m\geq1.$ Let
$\mathcal{A}$ be the set of all admissible pairs; that is%
\begin{align*}
\mathcal{A}  &  \mathcal{=}\left\{  \left(  z,u\right)  \in Y^{\alpha/2}%
\times\mathcal{U}_{\lambda}:\text{ }z\text{ is a weak solution of }\left(
\ref{5.1}\right)  \text{ satisfying }\left(  \ref{5.1b}\right)  \text{
corresponding }\right. \\
&  \left.  \text{to the critical value }\mathbf{c}_{k}\text{ of the functional
}F_{u,v}\text{ defined in }\left(  \ref{akcja1}\right)  \text{ for }%
u\in\mathcal{U}_{\lambda}\right\}  .
\end{align*}
It should be noted that Corollary \ref{Cor4.2} implies that the set of all
admissible pairs $\mathcal{A}$ is nonempty. In this section, our aim is to
find a pair $\left(  z_{u^{\ast}},u^{\ast}\right)  \in Y^{\alpha/2}%
\times\mathcal{U}_{\lambda}$ satisfying
\begin{equation}
J\left(  z_{u^{\ast}},u^{\ast}\right)  =\min_{\left(  z,u\right)
\in\mathcal{A}}J\left(  z,u\right)  . \label{optymalne}%
\end{equation}

The function $\Phi$ is required to satisfy:

\begin{enumerate}
\item[(C6)] the function $\Phi$ is measurable with respect to $x$ for any
$(z,u)\in\mathbb{R}\times U$ and continuous with respect to $(z,u)\in
\mathbb{R}\times U$ for a.e. $x\in\Omega$;

\item[(C7)] there exists $c>0$ such that
\[
\left\vert \Phi\left(  x,z,u\right)  \right\vert \leq c\left(  1+\left\vert
z\right\vert ^{s}\right)  ,
\]
for $z\in\mathbb{R},$ $u\in U$ and a.e. $x\in\Omega$. where $s\in\left[
1,2_{\alpha}^{\ast}\right)  $.
\end{enumerate}

The theorem on the existence of an optimal process to problem $\left(
\ref{optymalne}\right)  $ can now be formulated.

\begin{theorem}
If the functions $G$ and $\Phi$ satisfy conditions $\left(  C1\right)
-\left(  C7\right)  ,$ then optimal control problem $\left(  \ref{optymalne}%
\right)  $ possesses at least one optimal process $\left(  z_{u^{\ast}%
},u^{\ast}\right)  \in Y^{\alpha/2}\times\mathcal{U}_{\lambda}%
.\label{twierdzenie5.1}$
\end{theorem}

\begin{proof}
By $\left(  C6\right)  $ and $\left(  C7\right)  $ the cost functional is
well-defined and continuous with respect to $\left(  z,u\right)  \in
\mathbb{R}\times U$ variables. Let $\left(  z_{k},u_{k}\right)  ,$
$k\in\mathbb{N}$ be a minimizing sequence for problem $\left(  \ref{5.1}%
\right)  -\left(  \ref{5.2}\right)  $, i.e. $u_{k}\in\mathcal{U}_{\lambda},$
the equation $(-\Delta)^{\alpha/2}z_{k}\left(  x\right)  =G_{z}\left(
x,z_{k}\left(  x\right)  ,u_{k}\left(  x\right)  \right)  $ is satisfied in a
weak sense cf. $\left(  \ref{weak}\right)  $ and moreover $z_{k}=v$ in
$\mathbb{R}^{n}\backslash\Omega$ and%
\[
\lim_{k\rightarrow\infty}J\left(  z_{k},u_{k}\right)  =\inf_{\left(
z,u\right)  \in\mathcal{A}}J\left(  z,u\right)  .
\]
Entire class $\mathcal{U}_{\lambda}$ is equicontinuous and uniformly bounded,
so certainly $\left\{  u_{k}\right\}  $ is such as well. By
Arz\'{e}la-Ascoli's Theorem, there exists a subsequence $\left\{
u_{k}\right\}  $ such that $u_{k\text{ }}\rightarrow u_{0\text{ }}$uniformly
on $\Omega$ and $u_{0}\in\mathcal{U}_{\lambda}.$ By Corollary \ref{Cor4.2} the
sequence $\left\{  z_{k}\right\}  $, or at least some of its subsequence,
tends to $z_{0}$ in $Y^{\alpha/2}$ and $\left(  z_{0},u_{0}\right)
\in\mathcal{A},$ thus $J\left(  z_{0},u_{0}\right)  =\inf_{\left(  z,u\right)
\in\mathcal{A}}J\left(  z,u\right)  .$ It means that the process $\left(
z_{0},u_{0}\right)  $ is optimal for $\left(  \ref{optymalne}\right)  .$
\end{proof}

Analogous result for another class of controls $\mathcal{U}_{\Omega\left(
r\right)  }$ can be proved.\ By $\Omega\left(  r\right)  $\ we denote a fixed
decomposition of $\Omega$\ on $r$\ open subsets $\Omega_{i}$\ such that
$\bigcup\nolimits_{i=1}^{r}\Omega_{i}\subset\Omega$, $\mu\left(  \bigcup
_{i=1}^{r}\Omega_{i}\right)  =\mu\left(  \Omega\right)  $\ and $\Omega_{i}%
\cap\Omega_{j}=\emptyset$\ for $i\neq j,$\ $i,j=1,...,r$ with $\mu$ being the
Lebesgue measure on $\mathbb{R}^{n}.$\newline We shall say that a function
$u$\ is constant on $\Omega\left(  r\right)  $\ if $u$\ is constant on each
subset from decomposition $\Omega\left(  r\right)  ,$\ i.e. $u\left(
x\right)  =\mathrm{const}_{i}$\ for $x\in\Omega_{i},$\ $i=1,...,r.$ Finally
for the following class of controls
\[
\mathcal{U}_{\Omega\left(  r\right)  }=\left\{  u\in L^{\infty}\left(
\Omega,\mathbb{R}^{m}\right)  :u\left(  x\right)  \in U\text{ and }u\text{ is
constant on }\Omega\left(  r\right)  \right\}
\]
where $U$\ is a compact subset of $\mathbb{R}^{m},$ analogously to the proof
of Theorem \ref{twierdzenie5.1} one can prove the following theorem.

\begin{theorem}
If the functions $G$\ and $\Phi$\ satisfy conditions $\left(  C1\right)
-\left(  C7\right)  ,$\ then optimal control problem $\left(  \ref{optymalne}%
\right)  $\ possesses at least one optimal process $\left(  z_{u^{\ast}%
},u^{\ast}\right)  \in Y^{\alpha/2}\times\mathcal{U}_{\Omega\left(  r\right)
}.\label{twierdzenie5.2}$
\end{theorem}

\end{document}